\documentclass[11pt]{amsart}
 
\theoremstyle{plain}
\newtheorem{thm}{Theorem}[section]
\newtheorem{theorem}[thm]{Theorem}

\newtheorem{remark}[thm]{Remark}

\numberwithin{equation}{section}
\title[Comments on Sampson's approach]{Comments on Sampson's approach toward Hodge conjecture on Abelian varieties} 
\author{Tuyen Trung Truong}
\thanks{}
\subjclass[2010]{14Kxx, 14Cxx, 32xxx}
\date{\today}
\keywords{Abelian varieties; Hodge conjecture}
\address{School of Mathematics, Korea Institute for Advanced Study, Seoul 130-722, Republic of Korea}\email{truong@kias.re.kr}
\address{Current address: School of Mathematical Sciences, The University of Adelaide, SA 5005, Australia}\email{tuyen.truong@adelaide.edu.au}

\begin{document}

\maketitle

\begin{abstract}
\noindent Let $A$ be an Abelian variety of dimension $n$. For $0<p<2n$ an odd integer, Sampson constructed a surjective homomorphism $\pi :J^p(A)\rightarrow A$, where $J^p(A)$ is the higher Weil Jacobian variety of $A$. Let $\widehat{\omega}$ be a fixed form in $H^{1,1}(J^p(A),\mathbb{Q})$, and $N=\dim (J^p(A))$. He observes that if the map $\pi _*(\widehat{\omega }^{N-p-1}\wedge .): H^{1,1}(J^p(A),\mathbb{Q})\rightarrow H^{n-p,n-p}(A,\mathbb{Q})$ is injective, then the Hodge conjecture is true for $A$ in bidegree $(p,p)$. 

In this paper, we give some clarification of the approach and show that the map above is {not injective} except some special cases where the Hodge conjecture is already known. We propose a modified approach. 
\end{abstract}

%\tableofcontents

\section{Introduction and results}

A compact complex manifold $X$ is projective if it is a submanifold of a complex projective space $\mathbb{P}^{N}$. The Hodge conjecture is the following statement 

{\bf Hodge conjecture.} Let $X$ be a projective manifold. If $u\in H^{2p}(X,\mathbb{Q})\cap H^{p,p}(X)$ then $u$ is a linear combination with rational coefficients of the classes of algebraic cycles on $X$.

There have been a lot of works on the conjecture, however, it is still very largely open (see \cite{lewis}). The case of Abelian varieties, on which the cohomology groups are explicitly described, have been extensively studied, see Appendix 2 in \cite{lewis}. In this case, also, the Hodge conjecture is still open, even though many partial results have been obtained. 

Sampson \cite{sampson} (see also Appendix 2 in \cite{lewis}) proposed one approach toward proving the Hodge conjecture for Abelian varieties using Weil Jacobians. He suggested that the Hodge conjecture would follow if a certain map is injective. In this paper we show that in general this is not the case. The main idea is that instead of showing that the map is not injective, we show that the map is not surjective. It turns out that this conclusion is also valid for a more general class of surjective homomorphisms of Abelian varieties. We also give some clarification on the construction in Section 10 of his paper and propose a modified approach.   

We will first recall some basic definitions. 

\subsection{Abelian varieties}
 Let $A=V/L$ be an Abelian variety of dimension $n$. Here $V=\mathbb{R}^{2n}$ is equipped with a complex structure $J:V\rightarrow V$ with $J^2=-1$, and $L$ is a lattice of rank $2n$. There is one alternating bilinear form $E:V\times V\rightarrow \mathbb{R}$ such that $E(Jx,Jy)=E(x,y)$, $E(x,Jy)$ is a symmetric and positive definite bilinear form on $V$, and $E(L,L)\subset \mathbb{Z}$. There is associated an integral K\"ahler form on $A$, given by the following formula
\begin{eqnarray*}
\omega =\sum _{i,j}E(e_i,e_j)dx^i\wedge dx^j.
\end{eqnarray*}
Here $e_1,\ldots ,e_{2n}$ are a basis for $V$, and $x^i$ is the real coordinate corresponding to $e_i$. The K\"ahler form $\omega$ does not depend on the choice of the basis. 

There is also associated a Hermitian metric
\begin{eqnarray*}
H(x,y)=E(x,Jy)-iE(x,y).
\end{eqnarray*}

For more on Abelian varieties, see \cite{lange-birkenhake}.

\subsection{Weil Jacobians}
Let $e_1,\ldots ,e_{2n}$ be a basis for the lattice $L$. Let $0<p<2n$ be an odd integer. Define 
\begin{eqnarray*}
\widehat{V}=\bigwedge ^pV.
\end{eqnarray*}
We define $\widehat{L}\subset\widehat{V}$ to be the lattice generated by the elements $e_I=\wedge _{i\in I}e_i$, where $I$ is a multi-index of length $p$. 

$J$ defines a complex structure $\widehat{J}$ on $\widehat{V}$ by the formula $\widehat{J}(e_I)=\wedge _{i\in I}Je_i$. 

$E$ defines a bilinear form $\widehat{E}$ on $\widehat{V}$ by the formula: $\widehat{E}(e_I,e_J)=\det (E(e_i,e_j))_{i\in I,j\in J}$. 

It can then be checked that $\widehat{E}$ is alternating, $\widehat{E}(\widehat{J}{x},\widehat{J}{y})=\widehat{E}(x,y)$, $\widehat{E}(\widehat{L},\widehat{L})\subset \mathbb{Z}$, and $\widehat{E}(e_I,\widehat{J}e_J)$ is symmetric and positive definite. Thus $J^p(A)=\widehat{V}/\widehat{L}$ is an Abelian variety. 

There is an injection $f:H^{1,1}(J^p(A),\mathbb{Z})\rightarrow H^{p,p}(A,\mathbb{Q})$, see Proposition 7 in \cite{sampson}.

\subsection{Sampson's construction}

Starting from the K\"ahler form $\omega$ associated with the bilinear from, Sampson defines a surjective homomorphism $\pi :\widehat{V}\rightarrow V$, which is $\mathbb{C}$-linear and preserves the lattice $\widehat{L}$. Thus it descends to a homomorphism $\pi : J^p(A)\rightarrow A$. 

The construction of Sampson is to assign directly
\begin{eqnarray*}
\pi (e_I)=\sum _{j=1}^{2n}b_{I}^je_j, 
\end{eqnarray*}
where $b_I^j$ comes from the coefficients of the form 
\begin{eqnarray*}
\omega ^{(p+1)/2},
\end{eqnarray*}
and the inverse of the matrix $(E(e_i,e_j))$. Then he uses explicit computations to show that the map $\pi$ is surjective and $\mathbb{C}$-linear. 

{\bf First comment.} Here is our first comment.  If we consider what happens with the pullback map $\pi ^*: H^1(A,\mathbb{R})\rightarrow H^1(J^p(A),\mathbb{R})$, then the above construction will look more transparent. In fact, let $x^i$ be the coordinate corresponding to $e_i$, and $x^I$ the coordinate corresponding to $e_I$. Then we have
\begin{eqnarray*}
\pi (\sum _Ix^Ie_I)=\sum _{j=1}^{2n}(\sum _{I}b_I^jx^I)e_j.
\end{eqnarray*}
Hence $x^j=\sum _Ib_I^jx^I$. From this, we obtain
\begin{eqnarray*}
\pi ^*(dx^j)=\sum _Ib_I^jdx^I.
\end{eqnarray*}
Here we recall that given a basis $(v_j)$ for a vector space,    with corresponding coordinates $z^j$, then the form $dz^j$ is given by $dz^j(v_i)=\delta _i^j$.
 
Now we make the following identification $\psi :H^1(J^p(A),\mathbb{R})\rightarrow H^p(A,\mathbb{R})$. We assign $\psi (dx^I)=\wedge _{i\in I}dx^i$. Then, by using a quasi-symplectic basis $e_1,\ldots ,e_{2n}$ for $L$, we obtain a very simple formula
\begin{eqnarray*}
\psi \circ \pi ^*(dx^j)=cdx^j\wedge \omega ^{(p-1)/2}.
\end{eqnarray*}
Here $c$ is a non-zero constant. Thus we see that $\psi \circ \pi ^*$ is, up to a multiplicative constant, the Lefschetz map. 

By the Lefschetz isomorphism theorem (see Lecture 11 in \cite{lewis}), $\psi \circ \pi ^*$ is injective, and hence $\pi $ is surjective. The property that $\pi$ is $\mathbb{C}$-linear can also be checked by choosing the basis $Je_1,\ldots ,Je_{2n}$ in the definition of the map $\psi \circ\pi ^*$. 

\subsection{Non-surjectivity of the pushforward $\pi _*$}

Let the notation be as in the previous subsections. Sampson's proposed approach is as follows. Let $Z\subset J^p(A)$ be a subvariety of appropriate dimension. If the map
\begin{eqnarray*}
\iota :~\alpha \in H^{p,p}(A,\mathbb{Q})\mapsto f(\alpha )\in H^{1,1}(J^p(A),\mathbb{Q}) \mapsto \pi _*(f(\alpha ).Z)\in H^{p,p}(A,\mathbb{Q})
\end{eqnarray*}
is injective, then it is also surjective and the Hodge conjecture follows. 

{\bf A clarification.} By the Poincare theorem, the map $\pi$ is, up to isogeny, of the form $pr_A:~Ker(\pi )\times A'\rightarrow A$,  where the projection from $A'\rightarrow A$ is an isogeny. In Section 10 in \cite{sampson}, Sampson proposed to use $Z=Z'\times A'$, where $Z'$ is a subvariety of $Ker (\pi )$. However, we can see easily from a dimensional consideration that such a choice can not be appropriate. In fact, if $f(\alpha )$ is a hypersurface in $J^p(A)$ which intersects $Z'\times A'$, then $f(\alpha )$ will intersect $p\times A'$ for at least one point $p\in Z'$. But then $f(\alpha ).Z'\times A'$ has dimension at least $f(\alpha ).p\times A'$, and the latter has dimension at least $dim(A)-1$. Hence the projection of $f(\alpha ).Z'\times A'$ to $A$ contains a hypersurface of $A$, and is not of the desired codimension. 

From the discussion given in Section 10 in \cite{sampson}, and given that in general we do not know much about $H^{*,*}(A,\mathbb{Q})$ and $H^{*,*}(J^p(A),\mathbb{Q})$ except the existence of an ample class, a natural choice of $Z$ is to be the self-intersection of an ample class on $J^p(A)$. We now show that in this case, the map $\iota$ is in general not surjective, and hence it is also not injective (by dimensional considerations). While it is not easy to see directly whether the map $\iota $ is surjective or not (since the definitions of the maps $\pi$ and $f$ are highly transcendental), it turns out that the answer to a more general question is available. 

\begin{theorem}
Let $\pi :\widehat{A}=\widehat{V}/\widehat{L}\rightarrow A=V/L$ be a surjective homomorphism of Abelian varieties. Let $\widehat{\omega}$ be a fixed form in $H^{1,1}(\widehat{A},\mathbb{Q})$. Let $J$ and $\widehat{J}$ be the complex structures on $A$ and $\widehat{A}$. Let $N=\dim (\widehat{A})$ and $n=\dim (A)$. Let $q$ be any integer with $1\leq q\leq n$.

1) If $\widehat{\omega}(u,\widehat{J}u)\not= 0$ for all $0\not= u\in H_1(\widehat{A},\mathbb{R})$, then 
\begin{eqnarray*}
\dim (\pi _*(\widehat{\omega }^{N-q-1}\wedge H^{1,1}(\widehat{A},\mathbb{Q})))\leq \dim H^{1,1}(A,\mathbb{Q}).
\end{eqnarray*}

2) For any $\widehat{\omega}$
\begin{eqnarray*}
\pi _*(\widehat{\omega }^{N-q-1}\wedge H^{1,1}(\widehat{A},\mathbb{Q}))\subset \mathbb{R}\otimes _{\mathbb{Z}}\bigwedge ^{n-q}H^{1,1}(A,\mathbb{Q}).
\end{eqnarray*}

In particular, if $\dim H^{n-q,n-q}(A,\mathbb{Q})>\dim \bigwedge ^{n-q}H^{1,1}(A,\mathbb{Q})$, then the map 
\begin{eqnarray*}
\pi _*(\widehat{\omega }^{N-q-1}\wedge .):H^{1,1}(\widehat{A},\mathbb{Q})\rightarrow H^{n-q,n-q}(A,\mathbb{Q})
\end{eqnarray*}
is not surjective.
\label{PropositionMain}\end{theorem} 

\begin{proof}

1) Let $\widehat{A}=\widehat{V}/\widehat{L} $ and  $A=V/L$. Let $\widehat{E}$ be the  alternating form corresponding to $\widehat{\omega}$. We define $W\subset \widehat{V}$ to be the kernel of the map $\pi :\widehat{V}\rightarrow V$. Because the map $\pi$ is $\mathbb{C}$-linear, it follows that $\widehat{J}W=W$. Moreover, since $\pi$ is surjective, $\dim (W)=2N-2n$.

 We observe that if $\pi ^*(du)\in \pi ^*H^1(A,\mathbb{R})$ and $v\in W$, then $$\pi ^*(du)(v)=du (\pi (v))=du (0)=0.$$

We let $\widehat{W}^{\perp}$ to be the orthogonal complement of $W$, with respect to $\widehat{E}$. Because $\widehat{E}(x,\widehat{J}x)>0$ for all $0\not= x\in \widehat{V}$, we have $W\cap W^{\perp}=0$. Therefore, we have the decomposition
\begin{eqnarray*}
\widehat{V}=W\oplus W^{\perp}.
\end{eqnarray*}
We note that $\dim (W^{\perp})=2n$.

We choose a basis $e_1,\ldots e_{2N-2n}$ for $W$, and $f_1,\ldots ,f_{2n}$ a basis for $W^{\perp}$. We let $x^1,\ldots ,x^{2N-2n}$ and $y^1,\ldots ,y^{2n}$ be the corresponding coordinates. Then we have the corresponding $1$-forms $dx^1,\ldots ,dx^{2N-2n}$ and $dy^1,\ldots ,dy.^{2n}$ on $\widehat{V}$. 

By definition, we have
\begin{eqnarray*}
dy^j(e_i)=0
\end{eqnarray*}
for all $i,j$. Comparing with the above computations and taking dimensions into consideration, we conclude that $\pi ^*H^1(A,\mathbb{R})$ is generated by $dy^1,\ldots ,dy^{2n}$.

From the discussion above, the form 
\begin{eqnarray*}
\widehat{\omega}&=&\sum \widehat{E}(e_i,e_j)dx^i\wedge dx^j+\sum \widehat{E}(f_i,f_j)dy^i\wedge dy^j+\sum \widehat{E}(e_i,f_j)dx^i\wedge dy^j\\
&=& \sum \widehat{E}(e_i,e_j)dx^i\wedge dx^j+\sum \widehat{E}(f_i,f_j)dy^i\wedge dy^j
\end{eqnarray*}
has no cross terms. By point 5) we see that we can write $\widehat{\omega}=\omega _1+\omega _2$, where $\omega _1$ involves only $dx^i$, and $\omega _2=\pi ^*(\alpha )\in \pi ^*H^{2}(A,\mathbb{R})$.

Moreover, we see that $\omega _1$ is the restriction of $\widehat{\omega}$ to $W$, and $\pi ^*(\alpha )$ is the restriction of $\widehat{\omega}$ to $W^{\perp}$. Since $W$ and $W^{\perp}$ are both invariant under the complex structure $\widehat{J}$, both forms $\omega _1$ and $\pi ^*(\alpha)$ are of type $(1,1)$. Then $\alpha$ is of bidegree $(1,1)$ also.

We also have that both $\omega _1$ and $\alpha$ are rational. This again follows easily from the fact that $\omega _1$ and $\pi ^*(\alpha )$ are the restrictions of $\widehat{\omega}$ to $W$ and $W^{\perp}$, and both $\widehat{L}\cap W$ and $\widehat{L}\cap W^{\perp}$ have maximal ranks.

We now estimate the dimension of the image of the map $$\pi _*(\widehat{\omega }^{N-q-1}\wedge .): H^{1,1}(\widehat{A},\mathbb{Q})\rightarrow H^{n-q,n-q}(A,\mathbb{Q}).$$ Let $u_0\in H^{q,q}(X,\mathbb{R})$. Then, for any divisor $D\in H^{1,1}(\mathbb{A},\mathbb{Q})$ 
\begin{eqnarray*}
\pi _*(\widehat{\omega }^{N-q-1}\wedge  D)\wedge u_0=\pi _*(\widehat{\omega }^{N-q-1}\wedge \pi ^*(u_0)\wedge D).
\end{eqnarray*}

Using $\widehat{\omega} =\omega _1+\pi ^*(\alpha )$, we have
\begin{eqnarray*}
\widehat{\omega }^{N-q-1}\wedge \pi ^*(u_0)\wedge D&=&(\omega _1+\pi ^*(\alpha ))^{N-q-1}\wedge \pi ^*(u_0)\wedge D\\
&=&\sum _{j}c_{j}\omega _1^{N-q-1-j}\wedge \pi ^*(\alpha )^{j}\wedge \pi ^*(u_0)\wedge D.
\end{eqnarray*}
Here $c_j\in \mathbb{N}$ are constants. Then, we have that the $j$-th summand in the above sum is zero, unless $j+q\leq n$ and $N-q-1-j\leq N-n$. Hence there are only two terms left
\begin{eqnarray*}
\widehat{\omega }^{N-q-1}\wedge \pi ^*(u_0)\wedge D=c_1\omega _1^{N-n-1}\wedge D\wedge \pi ^*(\alpha ^{n-q}\wedge u_0)+c_2\omega _1^{N-n}\wedge D\wedge \pi ^*(\alpha ^{n-q-1}\wedge u_0).
\end{eqnarray*}

This shows that 
\begin{eqnarray*}
\pi _*(\widehat{\omega }^{N-q-1}\wedge  D)=\alpha ^{n-q-1}\wedge [\pi _*(c_1\omega _1^{N-n-1}\wedge D)\wedge \alpha +c_2\pi _*(\omega _1^{N-n}\wedge D)].
\end{eqnarray*}

We note that
\begin{eqnarray*}
\pi _*(c_1\omega _1^{N-n-1}\wedge D)\wedge \alpha +c_2\pi _*(\omega _1^{N-n}\wedge D)\in H^{1,1}(A,\mathbb{Q}).
\end{eqnarray*}
From this it follows that $\pi _*(\widehat{\omega }^{N-q-1}\wedge H^{1,1}(\widehat{A},\mathbb{Q}))$ is contained in the image of the linear map
\begin{eqnarray*}
\alpha ^{n-q-1}\wedge .:H^{1,1}(A,\mathbb{Q})\rightarrow H^{n-q,n-q}(A,\mathbb{Q}).
\end{eqnarray*}

Therefore
\begin{eqnarray*}
\dim _{\mathbb{Q}}\pi _*(\widehat{\omega }^{N-q-1}\wedge H^{1,1}(\widehat{A},\mathbb{Q}))\leq \dim _{\mathbb{Q}}H^{1,1}(A,\mathbb{Q}).
\end{eqnarray*}

2) Now we consider a general form $\widehat{\omega}\in H^{1,1}(\widehat{A},\mathbb{Q})$. We can write 
\begin{eqnarray*}
\widehat{\omega}=\lim _{t\rightarrow 0}\widehat{\omega}_t.
\end{eqnarray*}
Here for $t\not= 0$, then the bilinear form $\widehat{E}_t$ of $\widehat{\omega}_t$ satisfies the condition $\widehat{E}_t(x,\widehat{J}x)\not= 0$ for $0\not= x\in \widehat{V}$. Let $D\in H^{1,1}(\widehat {A},\mathbb{Q})$. We have
\begin{eqnarray*}
\pi _*(\widehat{\omega }^{N-q-1}\wedge D)=\lim _{t\rightarrow 0}\pi _*(\widehat{\omega}_t^{N-q-1}\wedge D).
\end{eqnarray*}
By the proof of 1), for each $t\not= 0$
\begin{eqnarray*}
\pi _*(\widehat{\omega}_t^{N-q-1}\wedge D)\in\bigwedge ^{n-q}H^{1,1}(A,\mathbb{Q}).
\end{eqnarray*}
Therefore, 
\begin{eqnarray*}
\pi _*(\widehat{\omega }^{N-q-1}\wedge D)\in \mathbb{R}\otimes _{\mathbb{Z}}\bigwedge ^{n-q}H^{1,1}(A,\mathbb{Q}).
\end{eqnarray*}

\end{proof}

\begin{remark}

The proof of Theorem \ref{PropositionMain} shows that even if we choose 
\begin{eqnarray*}
Z=\sum _j\widehat{\omega}_j^{N-q-1},
\end{eqnarray*}
the map proposed by Sampson is not injective in general.

The following modification of the approach, requiring  that
\begin{eqnarray*}
\pi _*(\wedge ^{N-q}H^{1,1}(J^p(A),\mathbb{Q}))=H^{n-q,n-q}(A,\mathbb{Q})
\end{eqnarray*}
may work.

\end{remark}


\begin{thebibliography}{XXXXX}

\bibitem{lange-birkenhake} H. Lange and C. Birkenhake, {\em Complex Abelian varieties}, Grund. der math. Wiss., volume 302, Springer-Verlag, 1992.

\bibitem{lewis} J. D. Lewis, {\em A survey of the Hodge conjecture}, 2nd edition, CRM monograph series, volume 10, American Mathematical Society, 1999. Appendix 2 there in is by B. B. Gordon.

\bibitem{sampson} J. H. Sampson, {\em Higher Jacobians and cycles on Abelian varieties}, Compositio Mathematica, tome 47, no 2, (1982), 133--147.





\end{thebibliography}
\end{document}